\documentclass[11pt] {article}

\usepackage{amsmath,amsthm,amssymb}
\usepackage{verbatim}
\usepackage[margin=1in]{geometry}

\newtheorem*{theorem1}{Classification Theorem} 

\newtheorem{theorem}{Theorem}[section]
\newtheorem{lemma}{Lemma}[section]

\newtheorem{proposition}{Proposition}[section]
\newtheorem{remark}{Remark}[section]
\newtheorem{definition}{Definition}[section]
\newcommand{\eqnsection}{
   \renewcommand{\theequation}{\thesection.\arabic{equation}}
   \makeatletter
   \csname @addtoreset\endcsname{equation}{section}
   \makeatother}
 
\def \be{\begin{equation}}
\def \ee{\end{equation}}
\def \bt{\begin{theorem}} 
\def \et{\end{theorem}}
\def \bl{\begin{lemma}} 
\def \el{\end{lemma}}
\def \bea{\begin{eqnarray}}
\def \eea{\end{eqnarray}}
\def \bas{\begin{eqnarray*}}
\def \eas{\end{eqnarray*}}

\def \wh{\widehat}

\def \Z{{\mathbb{Z}}}

\def \({\left(}
\def \){\right)}

\def \bc{\begin{center} }
\def \ec{\end{center} }
\def \bs{\begin{slide} }
\def \es{\end{slide} }

\def\square{{\vcenter{\vbox{\hrule height.3pt
        \hbox{\vrule width.3pt height5pt \kern5pt
           \vrule width.3pt}
        \hrule height.3pt}}}}
\def\qed{{\hfill $\square$ \bigskip}}

\eqnsection
\begin{document}
\def\wh{\widehat}
\def\ol{\overline}

\setlength{\unitlength}{1.3 cm}

\title{Fixation for coarsening dynamics in 2D slabs }
\author{Michael Damron\thanks{The research of M.D. was supported in part  by US NSF grant DMS-0901534} \\ 
\small{Indiana University, Bloomington} \\ 
\small{Princeton University} 
\and 
Hana Kogan \thanks {The research of H.K. was supported in part by US NSF grant OISE-0730136}\\ \small{Courant Institute, NYU} \\ 
\small{New York} \
\and 
C. M. Newman\thanks{The research of C.M.N. was supported in part by US NSF grants OISE-0730136
and DMS-1007524}  \\ \small{Courant Institute, NYU}
\\ 
\small{New York} \and 
Vladas Sidoravicius\thanks{The research of V.S. was supported in part by Brazilian CNPq grants 308787/2011-0 and  476756/2012-0 and Faperj grant E-26/102.878/2012-BBP}  \\ \small{IMPA, Rio de Janeiro}}
\maketitle

\begin{abstract}    
For the zero temperature limit of Ising Glauber Dynamics  on $2D$ slabs the
existence or nonexistence of vertices that do not fixate is determined as function
of slab thickness.
\end{abstract}

\section{Introduction}

In this paper, we study some natural questions concerning coarsening on
two-dimensional slabs and how the answers to those questions depend
on the width $k$ of the slab. Coarsening is a particular continuous time
Markov process (which is the zero-temperature limit of processes for
which the Ising model Gibbs distribution is stationary). The coarsening process,
which will be defined precisely below, is a particular type of majority
vote model, in which the state space is assignments of $\pm1$ to
the vertices of a (generally infinite) graph.
For the nearest neighbor
graph $\Z ^1$, it is exactly the standard voter model. We will
be interested in the case when the
initial distribution on the state space is i.i.d. product
measure with probability $p$
for a site to be $+1$.

Focusing first on the symmetric case, $p=1/2$, we note that
it is known that on $\Z^d$ with
$d = 1, \,2$, no sites fixate (almost surely). For $d=1$ this
is a result about the
standard one-dimensional voter model \cite{Ar}, while 
for $d=2$, the result
is somewhat more recent \cite{NNS}. For $d \geq 3$,
it is a wide open problem to
determine whether or not (and for which values of $d$) some
sites fixate; there
are some hints from the computational physics literature that fixation may
indeed occur for large enough $d$ \cite{St}, \cite{SKR}. See also
\cite{OKR} for interesting numerical results about non-fixation for the $d=3$
periodic cube.

Motivated by non-fixation for $d=2$ and the open $d=3$ problem,
in this paper we study
graphs that interpolate between $\Z^2$ and $\Z^3$ by considering width-$k$ slabs,
$S_k = \Z^2 \times \{0,1,\dots,k-1\}$, with free or periodic
boundary conditions in
the third coordinate. As we shall see, there is an interesting,
and somewhat unexpected,
dependence on the value of $k$. The main results of this paper are
that all sites fixate if and only if $2 \leq k \leq k_c^*$
(where the superscript $*$ is either $f$ or $p$ denoting free
or periodic boundary
conditions); for $k_c^* < k < \infty$, some sites fixate and some do not.
The critical widths are $k_c^f =2$ and $k_c^p =3$.

An announcement of these results and proofs for the simpler regions of
$k$-values appear in \cite{DKNS}. The most difficult cases of $k=3$ (periodic),
where all sites fixate,
and $k=4$ (free), where some sites do not fixate, are treated
in this paper. The proof
for $k=3$ (periodic) is of particular interest, because the
percolation theoretic
arguments may be of use in other settings. Two other
contributions of this paper
are (a) to show that the same results are valid for all $k$
for any $p \in (0,1)$,
and (b) to provide, in an appendix, a simpler analysis of the $k=2$ (periodic)
case that that given in \cite{DKNS}.

The fact that the results for slabs do not depend on $p \in (0,1)$ is also
true on $Z^1$ where there is never fixation. But this is not so in general,
as it has been proved in \cite{FSS}
that on $\Z^d$ with $d \geq 2$, all sites fixate at
$+1$ (resp., at $-1$) when $p$ is close enough to $1$ (resp., to $0$).
It is conjectured, and an important open problem to prove, that this
is in fact the case as long as $p \neq 1/2$.

We conclude our introductory remarks by mentioning two open problems
about $2D$ slabs. One is whether for those slabs with non-fixated sites,
these sites percolate rather than forming only finite connected components
surrounded by fixated sites? For small values of $k$ that seems unlikely, but
perhaps percolation can occur for large $k$. A second interesting question
is how the density of fixated sites behaves as $k \to \infty$.  In the free boundary condition setting,
one could consider the probability that the site at $(0,0,[(k-1)/2])$
fixates. If it vanishes in this limit,
that might supply a mechanism for proving that no fixation occurs in $\Z^3$. 



\subsection{Definitions}

The slab $S_k, ~k \geq 2$, is the graph with vertex set $\mathbb{Z}^2 \times \{0,1, \ldots, k-1\}$ 
and edge set $\mathcal{E}_k = \{\{x,y\} : \|x-y\|_1 = 1\}$. As is usual, we take an initial spin 
configuration $\sigma(0) = (\sigma_x(0))_{x \in S_k}$ on $\Omega_k = \{-1,1\}^{S_k}$ 
distributed using the product measure of $\mu_p,~ p \in (0,1)$, where
\[
\mu_p(\sigma_x(0)=+1) = p = 1- \mu_p(\sigma_x(0)=-1)\ .
\]

The configuration $\sigma(t)$ evolves as $t$ increases according to the zero-temperature 
limit of Glauber dynamics (the majority rule). To describe this, define the energy 
(or local cost function) of a site $x$ at time $t$ as 
\begin{equation}\label{eq: energy}
e_x(t) = - \sum_{y : \{x,y\} \in \mathcal{E}_k} \sigma_x(t)\sigma_y(t)\ .
\end{equation}
Note that up to a linear transformation, this is just the number of neighbors $y$ of $x$ 
such that $\sigma_y(t) \neq \sigma_x(t)$. Each site has an exponential clock with clocks 
at different sites independent of each other. When a site's clock rings, it makes an 
update according to the rules
\[
\sigma_x(t) = \begin{cases}
- \sigma_x(t^-) & \text{ if } e_x(t^-) > 0 \\
\pm 1 & \text{ with probability }1/2 \text{ if } e_x(t^-) = 0 \\
\sigma_x(t^-) & \text{ if } e_x(t^-) < 0
\end{cases}\ .
\]
Write $\mathbb{P}_p$ for the joint distribution of $(\sigma(0),\omega)$, the initial spins 
and the dynamics realizations.

The main questions we will address involve fixation. We say that the slab $S_k$ fixates 
for some value of $p$ if 
\[
\mathbb{P}_p(\text{there exists }T=T(\sigma(0),\omega) <
\infty \text{ such that } \sigma_0(t) = \sigma_0(T) \text{ for all } t \geq T) = 1\ .
\]
All of our results will hold for all $p \in (0,1)$, so we will write $\mathbb{P}$ for the 
measure $\mathbb{P}_p$. The setup thus far corresponds to the model with free 
boundary conditions; in the case of periodic boundary conditions, we consider 
sites of the form $(x,y,k-1)$ and $(x,y,0)$ to be neighbors in $S_k$. If $k=2$ 
then this enforces two edges between $(x,y,1)$ and $(x,y,0)$, so that in the 
computation of energy of a site, that neighbor counts twice.


\subsection{Main results}
Let $p \in (0,1)$ be arbitrary.

\begin{theorem}\label{02}
With periodic boundary conditions, all sites in $S_3$ fixate.
\end{theorem}
\noindent The proof of Theorem~\ref{02} is given in Section~\ref{sec: thm_3}, using the results of 
Sections~\ref{sec: prelim} and \ref{sec: second_prop}.
In the appendix we give 
a simplified proof for fixation of all sites in $S_2$ with periodic boundary conditions. It does 
not use a comparison to bootstrap percolation (as in \cite{DKNS}) and therefore 
should allow for more general initial measures for $\sigma(0)$.

\begin{theorem} \label{03}
With periodic boundary conditions some sites in $S_4$ do not fixate.
\end{theorem}

\noindent The proof of Theorem \ref{03} is given in Section~\ref{sec: thm_4}. 
Results for free boundary conditions (for all $k$) or periodic conditions for $k=2$ 
and $k\geq 5$ were proved in \cite{DKNS}  for $p=1/2$. It is straightforward   to see 
that the proofs in  \cite{DKNS}  extend to all $p\in (0,1)$.    Combined with the preceding two theorems, we 
have the following  complete characterization of slabs, where we say that $S_k$ fixates if all sites in $S_k$
fixate.

\begin{theorem1} \label {01}
With free boundary conditions, $S_k$ fixates if and only if $k = 2$. With periodic 
boundary conditions, $S_k$ fixates if and only if $k \leq 3$.
\end{theorem1}

\medskip

\noindent {\bf Remark.} It is an elementary fact that for $k\geq 2$ and either free or periodic boundary conditions, 
some sites fixate. Thus in all cases where $S_k$ does not fixate, there are sites that fixate and sites that do not.


\section{Preliminary results}\label{sec: prelim}

We will need to develop some terminology and recall some results before proceeding. 
All results hold for slabs with periodic boundary conditions unless stated otherwise. 
We say a vertex $v$ \emph{flips} at time $t$ if $\sigma_v(t^-) \neq \sigma_v(t)$.

\begin{definition}
A vertex is called variable in the realization $(\sigma(0),\omega)$ if it flips infinitely 
many times. We call a flip at time $t$ of a vertex $v$ energy lowering if $e_v(t^-)>0$.
\end{definition}
Note that if a vertex flips infinitely many times then the set of times at which it flips 
is almost surely unbounded. This follows from the fact that the waiting time for 
clocks has a non-degenerate distribution. The following lemma is proved in \cite{NNS} 
in some generality and applies to the slab $S_k$ for any $k$.

\bl \label{26}
Any vertex in $S_k$ has almost surely only finitely many energy lowering flips.
\el

For the proofs of the main results, we need the notion of stability. 
\begin{definition}
The vertex $v$ is called unstable in $\sigma \in \Omega_k$ if 
$ -\sum_{w : \{v,w\} \in \mathcal{E}_k} \sigma_v\sigma_w \geq 0$. Otherwise $v$ is stable. 
A vertex is stable (unstable) at time $t$ if it is stable (unstable) in $\sigma(t)$.
\end{definition}
\noindent We make the following observation regarding stability. For the statement, 
we say that an event $A \subset \Omega_k$ \emph{occurs infinitely often} if the 
set of times $\{t : \sigma(t) \in A\}$ is unbounded.

\bl\label{lem: stability}
Let $v \in S_k$. With probability one the following statements hold.
\begin{enumerate}
\item There exists $T = T(v,\sigma(0),\omega)$ such that 
$-\sum_{w : \{v,w\} \in \mathcal{E}_k} \sigma_v(t)\sigma_w(t) \leq 0$ for all $t \geq T$.
\item $v$ is variable if and only if it is unstable infinitely often if and only if it has 
exactly three same sign neighbors infinitely often.
\end{enumerate}
\el
\begin{proof}
For the proof, we use the following simple lemma from \cite{DKNS}:

\bl \label{22} Let $A$, $B$ be cylinder events in $\Omega_k$. If 
\[
\inf_{\sigma \in A} \mathbb{P}\left(\sigma(t) \in B \text{ for some } t \in [0,1) ~\big|~ \sigma(0) = \sigma\right)>0\ ,
\] 
then
\[
\mathbb{P}\left(B \text{ occurs infinitely often} ~\big|~ A \text{ occurs infinitely often}\right) = 1\ .
\]
\el
If with positive probability $e_v(t) > 0$ infinitely often, then an application of 
Lemma~\ref{22} shows that $v$ has infinitely many energy-lowering flips 
with positive probability, contradicting Lemma~\ref{26}. This proves the 
first statement.

Next, if $v$ is variable, it must flip infinitely often and, letting $t$ be one 
time at which $v$ flips, $\sigma_v(s) \geq 0$ for all $s<t$ sufficiently close 
to $t$. This means $v$ is unstable infinitely often. Conversely, if $v$ is 
unstable infinitely often, $e_v(t) \geq 0$ infinitely often, and using Lemma~\ref{22} 
with $A = \{v \text{ is unstable}\}$ and $B = \{\sigma_v = +1\}$ or 
$B = \{\sigma_v = -1\}$ we see that $v$ flips infinitely often.

If $v$ is unstable infinitely often then $e_x(t) \geq 0$ infinitely often. 
By the first statement of this lemma, $e_v(t) = 0$ infinitely often, and 
$v$ has exactly three same sign neighbors infinitely often. Conversely, if
$v$ has exactly three same sign neighbors infinitely often then it is clearly 
unstable infinitely often.
%
\end{proof}


\section{Fixed columns are monochromatic}

For $(x,y) \in \mathbb{Z}^2$ we write $C_{x,y}$ for the column of vertices at coordinate $(x,y)$:
\[
C_{x,y} = \{(x,y,i) : i = 0, \ldots, k-1\}\ .
\]
A column $C$ is \emph{monochromatic} in $\sigma \in \Omega_k$ if $\sigma_v = \sigma_w$ for all $v,w \in C$. A realization $(\sigma(0),\omega)$ of initial configuration and dynamics is said to be \emph{eventually in } $A \subset \Omega_k$ if there is some $T=T(\sigma(0),\omega,A)$ such that if $t \geq T$ then $\sigma(t) \in A$. We say a column \emph{flips finitely often} if each of its vertices flips finitely often.

In this section we prove the following.
\begin{proposition} \label{11}
With probability one, if a column in $S_3$ flips finitely often then it is eventually monochromatic.
\end{proposition}

We note that this result is parallel to the one for $S_2$ shown in \cite{DKNS}: all columns in a slab are eventually monochromatic if and only if all sites in the slab flip finitely often. However we will see in Section~\ref{sec: thm_4} that it fails for $S_4$.

\begin{proof}
The proof will proceed by contradiction, so assume that with positive probability there is a column that flips finitely often but is not eventually monochromatic. It must then have a terminal state; that is, the spins at vertices in this column have a limit as $t\to \infty$. This limit is assumed to be non-monochromatic, so we begin the analysis by defining a site percolation process on $\mathbb{Z}^2$ corresponding to certain non-monochromatic sites. Given a configuration $\sigma \in \Omega_k$ and $(x,y) \in \mathbb{Z}^2$, we say that the column $C_{x,y}$ is type-1 in $\sigma$ if
\[
\sigma_{(x,y,0)} = \sigma_{(x,y,1)} =+1 \text{ but } \sigma_{(x,y,2)} = -1\ .
\]
and type-2 if $\sigma_{(x,y,0)}=\sigma_{(x,y,2)}=+1$ but $\sigma_{(x,y,1)}=-1$. We define $\eta = \eta(\sigma) \in \{0,1,2\}^{\mathbb{Z}^2}$ by
\[
\eta_{(x,y)} = \begin{cases}
1 & \text{ if } C_{x,y} \text{ is type-1 in } \sigma \\
2 & \text{ if } C_{x,y} \text{ is type-2 in } \sigma \\
0 & \text{ otherwise}
\end{cases}\ .
\]
If $\eta_{(x,y)}=r$ then we say $(x,y)$ is type-$r$ in $\eta$. The pair $(\sigma(0),\omega)$ induces a configuration $\eta(t) = \eta(\sigma(t))$. Let 
\[
A_r(x,y) = \{C_{x,y} \text{ is eventually type-} r\} \text{ for } r = 1, 2\ .
\]
By the assumption that there exist columns that flip finitely often but are not eventually monochromatic, we must have either $\mathbb{P}(A_r(x,y))>0$ for some $r$ and all $(x,y)$ or the corresponding statement with a global flip; that is $\mathbb{P}(B_r(x,y))>0$ for some $r$ and all $(x,y)$, where $B_r(x,y)$ is the event that $C_{x,y}$ is eventually type-$r$ in the configuration $\eta(-\sigma(t))$, induced by the global flip $-\sigma(t)$. Both cases are handled identically, so we will assume here that $\mathbb{P}(A_r(x,y))>0$ for some $r$.

By spatial symmetry, we must then have $\mathbb{P}(A_r(x,y))>0$ for all $r$ and all $(x,y)$. By translation-ergodicity of the model there are almost surely infinitely many values of $n \in \mathbb{N}$ such that each $A_r(n,0)$ occurs. It follows that there exist $M_0,N_0 \in \mathbb{N}$ such that
\begin{equation}\label{eq: assumption_1}
\mathbb{P}(A_1(0,0) \cap A_2(M_0,0) \cap A_2(0,N_0)) > 0\ .
\end{equation}

Next we recall the notion of $*$-connectedness: two vertices $w,z \in \mathbb{Z}^2$ are \emph{neighbors} if $\|w-z\|_1 = 1$ and are $*$\emph{-neighbors} if $\|w-z\|_\infty = 1$. A \emph{path} is a sequence of vertices $(w_1, \ldots, w_k)$ such that $w_i$ and $w_{i+1}$ are neighbors for $i=1, \ldots, k-1$ and a $*$\emph{-path} is a sequence such that $w_i$ and $w_{i+1}$ are $*$-neighbors for $i=1, \ldots, k-1$. Given a realization $\eta \in \{0,1,2\}^{\mathbb{Z}^2}$ and $r \in \{0,1,2\}$, the \emph{$r$-cluster} ($r*$\emph{-cluster}) of a vertex $z$ is the set of vertices of type-$r$ which are connected to $z$ by a path ($*$-path) all of whose vertices are type-$r$. Note that if $z$ is not of type-$r$ then both its $r$-cluster and $r*$-cluster are empty. Two sets $V,U \subset \mathbb{Z}^2$ are $r$-connected ($r*$-connected), written $V \rightarrow_r U$ ($V \rightarrow_{r*}U$) if there are vertices $v \in V$ and $u \in U$ such that the $r$-cluster ($r*$-cluster) of $v$ contains $u$. If this connection can be made using only vertices in some set $D$, we say $U$ and $V$ are $r$-connected ($r*$-connected) in $D$ and write $U \xrightarrow{D}_r V$ ($U \xrightarrow{D}_{r*} V$).

In the box $B = \{0, \ldots, M_0\} \times \{0, \ldots, N_0\}$, write $L = \{0\} \times \{0, \ldots, N_0\}$, $R=\{M_0\} \times \{0, \ldots, N_0\}$, $D=\{0, \ldots, M_0\} \times \{0\}$ and $U=\{0, \ldots, M_0\} \times \{N_0\}$ for the left, right, lower and upper sides respectively. We note the following property of $r*$-clusters of $\eta(t)$ in this box:
\bl \label{lem: star_connection}
Assume \eqref{eq: assumption_1} and let $E$ be the event (in $\{0,1,2\}^{\mathbb{Z}^2}$) that $(0,0)$ is not $1*$-connected in $B$ to $R \cup U$. Then 
\[
\mathbb{P}(A_1(0,0) \text{ but } \eta(t) \in E \text{ infinitely often}) > 0\ .
\]
\el
\begin{proof}
Let $\eta$ a configuration such that the following three conditions hold: 
\begin{equation}\label{eq: bad_condition}
(0,0) \xrightarrow{B}_{1*} R \cup U,~(M_0,0) \xrightarrow{B}_{2*} L \cup U \text{ and } (0,N_0) \xrightarrow{B}_{2*} R \cup D
\end{equation}
and write $\tilde B$ for the set of type-2 vertices in $B$. By planarity, the last two conditions ensure that the connected component $C$ of $(0,0)$ in $B \setminus \tilde B$ does not intersect $R \cup U$. So write $B$ as a disjoint union $C \cup \tilde B \cup \hat C$, where $\hat C$ is defined as $B \setminus (C \cup \tilde B)$. Note that because $C$ is a maximal connected subset of $B \setminus \tilde B$, $C$ does not contain a vertex that is a neighbor of $\hat C$.

The first condition in \eqref{eq: bad_condition} gives a $*$-connected path $P$ (with vertices written in order as $x_0, \ldots, x_n$) from $(0,0)$ to $R \cup U$ in $B$ all of whose vertices are type-1. Because $x_0 \in C$ and $x_n \notin C$ there exists a first $i$ such that $x_i \in C$ but $x_{i+1} \notin C$. Both of these vertices are type-1, so they cannot be in $\tilde B$. This means $x_i \in C$ but $x_{i+1} \in \hat C$, so they are not neighbors, they are $*$-neighbors. Both of these vertices share a $2\times 2$ block with two other vertices $a$ and $b$. If either of $a$ or $b$ were in $C$ (or $\hat C$), they would be neighbors to a vertex in $\hat C$ (or $C$), a contradiction, so they must be in $\tilde B$ and thus type-2. Summarizing, we can find a vertex $v \in \{0, \ldots, M_0-1\} \times \{0, \ldots, N_0-1\}$ such that either
\begin{enumerate}
\item $\eta_v = \eta_{v+(1,1)} = 1$ and $\eta_{v+(1,0)} = \eta_{v+(0,1)} = 2$ or
\item $\eta_v = \eta_{v+(1,1)} = 2$ and $\eta_{v+(1,0)} = \eta_{v+(0,1)} = 1$.
\end{enumerate}
If $\eta(t)$ satisfies \eqref{eq: bad_condition} then in the first case above, writing $v=(x,y)$, the vertex $(x,y,2)$ has a negative spin but at least 4 neighbors with positive spin, giving it an opportunity for an energy-lowering flip. A similar statement holds in the second case. Therefore Lemmas~\ref{22} and \ref{26} show that almost surely, $\eta(t)$ satisfies \eqref{eq: bad_condition} only finitely often.

On the event $A_1(0,0) \cap A_2(M_0,0) \cap A_2(0,N_0)$ in \eqref{eq: assumption_1}, almost every configuration must fail to satisfy at least one of the conditions in \eqref{eq: bad_condition}. Therefore at least one of the following three events has positive probability: 
\begin{enumerate}
\item $A_1(0,0) \cap \left\{ \eta(t) \in \{(0,0) \xrightarrow{B}_{1*} R \cup U\}^c \text{ infinitely often}\right\}$
\item $A_2(M_0,0) \cap \left\{ \eta(t) \in \{(M_0,0) \xrightarrow{B}_{2*} L \cup U\}^c \text{ infinitely often}\right\}$
\item $A_2(0,N_0) \cap \left\{ \eta(t) \in \{(0,N_0) \xrightarrow{B}_{2*} R \cup D\}^c \text{ infinitely often}\right\}$.
\end{enumerate}
However spatial symmetry of $\mathbb{P}$ implies that these events have the same  probability. Therefore the first has positive probability and we are done. 
\end{proof}

The previous lemma imposes a certain restriction on the geometry of the $1*$ cluster of $(0,0)$ in $B$. To take advantage of this, we consider minimal clusters. Given $r \in \{0,1,2\}$, we say that a set $V \subset B$ is a \emph{recurrent $r$-cluster} in $B$ for the realization $(\sigma(0),\omega)$ if, infinitely often, $V$ is the intersection of an $r$-cluster of $\eta(t)$ with $B$. $V \subset B$ is a \emph{minimal recurrent $r$-cluster} if it is a recurrent $r$-cluster but no proper subset of $V$ is. On the event in Lemma~\ref{lem: star_connection}, there is a recurrent cluster in $B$ that contains $(0,0)$ but no point of $R\cup U$, so there is a minimal such cluster. Because there are only finitely many clusters in this box, we can fix $V \subset B$ such that $V$ contains $(0,0)$, $V$ does not intersect $R \cup U$ and 
\begin{equation}\label{eq: V_condition}
\mathbb{P}(V \text{ is a minimal recurrent }1 \text{-cluster}) > 0\ .
\end{equation}

$V$ must contain a vertex $v$ such that 
\begin{equation}\label{eq: v_condition}
v \in V \text{ but } v+(1,0),~v+(0,1) \text{ and } v+(1,1) \notin V\ .
\end{equation}
To see this, choose any vertex $v \in V$ with maximal $\ell_1$ norm. Since $V$ is finite, such a $v$ exists. Because $V$ does not intersect $R \cup U$, the vertices $v+(1,0)$, $v+(0,1)$ and $v+(1,1)$ must be in $B$. Since they have $\ell_1$ norm larger than that of $v$, they also cannot be in $V$.

The following lemma will contradict \eqref{eq: V_condition} and complete the proof of Proposition~\ref{11}.
\bl \label{25}
Let $V \subset B$ be such that \eqref{eq: v_condition} holds for some $v$. Then
\[
\mathbb{P}(V \text{ is a minimal recurrent }1 \text{-cluster})=0\ .
\]
\el
\begin{proof}
Assume the probability is positive and consider a configuration $(\sigma(0),\omega)$ in which $V$ is a minimal recurrent 1-cluster. Define $t_1 < t_2 < \cdots$ to be a (random) sequence of times with $t_n \to \infty$ such that $V$ is the intersection of a 1-cluster of $\eta(t)$ with $B$. An application of Lemma~\ref{22} along with minimality shows that for all large $n$, all sites in the column $C_v$ are stable in $\sigma(t_n)$. Therefore, writing $v=(x,y)$, the vertex $(x,y,2)$, which has a negative spin and two neighbors in $C_v$ with positive spin, must have all other neighbors with negative spin. This implies that for all large $n$, the spins at $(x+1,y,2)$ and $(x,y+1,2)$ are negative in $\sigma(t_n)$. Furthermore, these vertices must be stable for all large $n$, lest at least one flips to +1 and makes $(x,y,2)$ unstable. By (3.4) the remaining spins of $C_{(x+1, y)}$ and $C_{(x, y+1)}$ can not both be positive. Hence for all large $n$, at least one other vertex in each of $C_{(x+1,y)}$ and $C_{(x,y+1)}$ must have negative spin.

Stability of vertices $(x,y,0)$ and $(x,y,1)$ (both of which have positive spin) gives now that only one vertex of each of $C_{(x+1,y)}$ and $C_{(x,y+1)}$ with third coordinate not equal to 2 can have negative spin. In addition, they must have different third coordinate. Using symmetry, we have now argued that if the lemma fails, then with positive probability, the configuration pictured in Figure~\ref{fig: minimal_fig} occurs in $\sigma(t_n)$ for an increasing sequence $(t_n)$ growing to infinity.

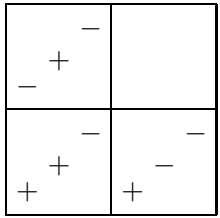
\begin{figure}[h]
\begin{center}
\setlength{\unitlength}{1.4cm}
\begin{picture}(2, 2.5)
\multiput(0,0)(0,1){3}{\line(1,0){2}}
\multiput(0,0)(1,0){3}{\line(0,1){2}}
\multiput(.1, .15)(1,0){2}{$+$}\multiput(.4, .4)(0,1){2}{$+$}
\multiput(.7,.7)(1,0){2}{$-$}\put(.7,1.7){$-$}
\put(.1,1.15){$-$}\put(1.4,.4){$-$}
\end{picture}
\caption{Depiction of the configuration near a corner in a minimal recurrent 1-cluster at a large time. The bottom left box represents the column $C_v$ and the spins are listed for vertices in this column in increasing third coordinate.}
\label{fig: minimal_fig}
\end{center}
\end{figure}

In the above configuration, for large $n$, we again invoke stability, but of the vertices $(x+1,y,0)$ and $(x,y+1,1)$ with positive spin. This implies that the vertices $(x+1,y+1,0)$ and $(x+1,y+1,1)$ must have positive spin. However, at these $t_n$'s, the column $C_{(x+1,y+1)}$ is not type-1, so the spin at $(x+1,y+1,2)$ is $+1$, and the column is monochromatic. We have now reached a contradiction: at each of these times, a finite sequence of flips can force the minimal cluster to shrink. The spin at $(x+1,y,1)$ can flip to +1, followed by the spin at $(x+1,y,2)$ and then the spin at $(x,y,2)$. Applying Lemma~\ref{22} completes the proof.
\end{proof}
Under assumption \eqref{eq: assumption_1}, we derived inequality \eqref{eq: V_condition}. The contradiction given by Lemma~\ref{25} implies that \eqref{eq: assumption_1} must have been false and we are done.
\end{proof}

\section{Fixed columns proliferate}\label{sec: second_prop}
In this section we prove  that the neighbors of fixed columns are fixed.

\begin{proposition}\label{12}
Let $u,v \in \mathbb{Z}^2$ be neighbors. With probability one, if $C_u$ flips finitely often, then so does $C_v$.
\end{proposition}

An event $A \subset \Omega_k$ is \emph{eventually absent} (or e. absent) if $\mathbb{P}(\sigma(t) \in A \text{ infinitely often})=0$.
\bl \label{30}
Let $v, w \in S_k$ be neighbors for $k \geq 2$. The event $\{v$ and $w$ are unstable$\} \cap \{\sigma_v=\sigma_w\}$ is e-absent.
\el

\begin{proof}
For a contradiction, suppose that with positive probability, the event that both $v$ and $w$ are unstable and $\sigma_v=\sigma_w$ occurs infinitely often. At any one of these times, $v$ has a chance to flip. If $v$ flips but no clocks assigned to any other vertex within distance 2 of $v$ ring beforehand, then $w$ would have at least 4 opposite sign neighbors. Therefore we can apply Lemma~\ref{22} to deduce that with positive probability, $-\sum_{z : \{w,z\} \in \mathcal{E}_k} \sigma_w\sigma_z > 0$ infinitely often. This contradicts part 1 of Lemma \ref{lem: stability}.
\end{proof}

The following two lemmas are used repeatedly in the proof. A column $C$ is called \emph{positive} if the spins of all its vertices equal $+1$ and called \emph{negative} if the spins are $-1$.

\bl \label{31}

For $x,y \in \mathbb{Z}$, let $A_{(x,y)} \subset \Omega_3$ be the event defined by the conditions
\begin{enumerate}
\item $C_{(x,y)}$ is positive but at least one of $C_{(x+1,y)}, C_{(x,y+1)}, C_{(x+1,y+1)}$ is not,
\item for some $i,j \in \{0, 1, 2\}, \sigma_{(x+1, y, i)} = \sigma_{(x, y+1, j)} = +1$ and
\item $\sigma_{(x+1, y+1, k)}=+1$ for some $k \in \{0,1,2\} \setminus \{i,j\}$. 
\end{enumerate}
(See Figure~\ref{fig: lem31}.) Then $A_{(x,y)}$ is e. absent.
\el

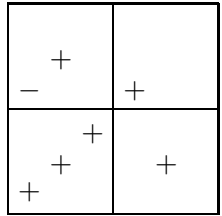
\begin{figure}[h]
\begin{center}
\setlength{\unitlength}{1.4cm}
\begin{picture}(2, 2.5)
\multiput(0,0)(0,1){3}{\line(1,0){2}}
\multiput(0,0)(1,0){3}{\line(0,1){2}}
\put(.1,.15){$+$}
\put(.4,.4){$+$}
\put(.7,.7){$+$}
\put(.1,1.1){$-$}
\put(.4,1.4){$+$}
\put(1.4,.4){$+$}
\put(1.1,1.1){$+$}
\end{picture}
\caption{Illustration of the event $A_{(x,y)}$ in Lemma~\ref{31} with $i=j=1,k=0$. The bottom left column is $C_{(x,y)}$. All unmarked spins are unspecified.}
\label{fig: lem31}
\end{center}
\end{figure}

\begin{remark}\label{rem: neg31}
By identical reasoning, Lemma~\ref{31} holds for the global flip (all positive spins become negative and vice-versa) of $A_{(x,y)}$.
\end{remark}

\begin{proof}
By way of contradiction, suppose that with positive probability, $A_{(x,y)}$ occurs infinitely often. When $A_{(x,y)}$ occurs, the vertices $(x+1, y, k)$ and $(x, y+1, k)$ have at least three positive neighbors each, so they are unstable. If the spins at these two vertices flip to $+1$ (without any other clocks ringing for vertices in $C_{(x,y)}, C_{(x+1,y)}, C_{(x,y+1)}$ or $C_{(x+1,y+1)}$) then all other spins of vertices in $C_{(x+1,y)}$ and $C_{(x,y+1)}$ have at least three positive neighbors, and can flip. Continuing, we see that there is a finite sequence of clock rings and spin flips that force $B_{(x,y)}$ to occur, where $B_{(x,y)}$ is the event such that $C_{(m,n)}$ is positive for $m=\{x, x+1\}$ and $n=\{y,y+1\}$. Therefore,
\[
\inf_{\sigma \in A_{(x,y)}} \mathbb{P} ( \sigma(t) \in B_{(x,y)} \text{ for some } t \in [0,1) \mid \sigma(0) = \sigma) >0\ .
\]
By Lemma~\ref{22}, $B_{(x,y)}$ occurs infinitely often with positive probability. Since this event is also absorbing (that is, $\mathbb{P}( \sigma(t) \in B_{(x,y)} \text{ for all } t \geq 0 \mid \sigma(0) \in B_{(x,y)})=1$), the event $A_{(x,y)}$ is e. absent, a contradiction.
\end{proof}


\bl \label{32}
Let $A$ be the event that $C_{(1,1)}$ is positive, $C_{(2,2)}$ is negative, $C_{(1,2)}$ is monochromatic and $C_{(2,1)}$ contains an unstable vertex. Then $A$ is e. absent.
\el

\begin{figure}[h]
\begin{center}
\setlength{\unitlength}{1.4cm}
\begin{picture}(2, 2.5)
\multiput(0,0)(0,1){3}{\line(1,0){2}}
\multiput(0,0)(1,0){3}{\line(0,1){2}}
\put(.1,.15){$+$}
\put(.4,.4){$+$}
\put(.7,.7){$+$}
\put(.1,1.15){$+$}
\put(.4,1.4){$+$}
\put(.7,1.7){$+$}
\put(1.1,.1){$u$}
\put(1.1,1.15){$-$}
\put(1.4,1.4){$-$}
\put(1.7,1.7){$-$}
\end{picture}
\caption{Illustration of the event $A$ in Lemma~\ref{32}. The bottom left column is $C_{(1,1)}$ and $C_{(1,2)}$ is monochromatic (pictured here as $+1$). The column $C_{(2,1)}$ has an unstable spin (marked $u$) at vertex $(2,1,0)$.}
\label{fig: lem32}
\end{center}
\end{figure}
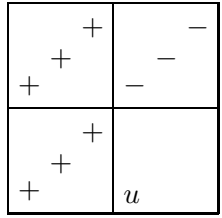

\begin{proof}
Suppose that with positive probability, $A$ occurs infinitely often. At least one of $A_0:=A\cap\{C_{(1,2)}\text{ is positive}\}$ and $A \cap \{C_{(1,2)}\text{ is negative}\}$ occurs infinitely often with positive probability. We will assume that $A_0$ does, as the following reasoning is identical in the other case. By part 1 of Lemma~\ref{lem: stability}, with probability one, any given vertex must have exactly three positive and three negative neighbors at any large time at which it is unstable. Using permutation invariance of the different levels of the slab, the event $A_1:=A_0\cap \{(2,1,0) \text{ has three positive and three negative neighbors}\}$ occurs infinitely often with positive probability. 

We will consider two cases depending on the status of spins in $C_{(2,1)}$. To do so, we define 
\[
A_2^+ := A_1 \cap \{C_{(2,1)} \text{ is positive}\} \text{ and } A_2^- := A_1 \cap \{C_{(2,1)} \text{ is negative}\}\ .
\]
We claim that at least one of $A_2^+$ or $A_2^-$ occurs infinitely often with positive probability. By way of contradiction, assume this is false, so that almost surely, for all large times, if $A_2$ occurs, then two vertices of $C_{(2,1)}$ have spins opposite of that of the third. The third 
spin must be unstable and at large times has exactly three positive and three negative neighbors. It has a positive probability to flip, so Lemma~\ref{22} implies that almost surely, it will infinitely often, leaving $C_{(2,1)}$ monochromatic and this spin still unstable. Permutation invariance of the levels of the slab implies that $A_2^+ \cup A_2^-$ occurs infinitely often with positive probability.

{\bf Case 1}. $A_2^+$ occurs infinitely often with positive probability. When $A_2^+$ occurs, $C_{(2,1)}$ is positive and $(2,1,0)$  has equal number of positive and negative neighbors, so the spins at $(3,1,0)$ and $(2,0,0)$ must be negative. If one of the vertices of $C_{(2,2)}$ is unstable then it may flip, leading to a configuration in $C_{(1,1)}, C_{(2,1)}, C_{(1,2)}$ and $C_{(2,2)}$ from the event $A_{(1,1)}$ in Lemma~\ref{31}. Therefore, by Lemma~\ref{22}, the event $A_3^+:= A_2^+ \cap \{\text{all vertices of }C_{(2,2)}\text{ are stable}\}$ occurs infinitely often with positive probability. This means that when $A_3^+$ occurs, the column $C_{(3,2)}$ is negative. This results in the sign distribution pictured in Figure~\ref{fig: 32fig_a}.

\begin{figure}[h]
\begin{center}
\setlength{\unitlength}{1.4cm}
\begin{picture}(3,3.5)
\multiput(0,0)(1,0){4}{\line(0,1){3}}
\multiput(0,0)(0,1){4}{\line(1,0){3}}
\multiput(.2, 1.2)(0,1){2}{$+$} \multiput(.45, 1.45)(0,1){2}{$+$} \multiput(.7, 1.7)(0,1){2}{$+$}
\put(1.2, 1.2){$+$} \put(1.45, 1.45){$+$} \put(1.7, 1.7){$+$}
\put(1.2, .2){$-$}
\put(2.2, 1.2){$-$}
\multiput(1.2, 2.2)(1,0){2}{$-$}
\multiput(1.45, 2.45)(1,0){2}{$-$}\multiput(1.7, 2.7)(1,0){2}{$-$}
\put(1.2, .2){$-$}\put(2.2, 1.2){$-$}
\end{picture}
\caption{Illustration of the event $A_3^+$ in Lemma~\ref{32}. All unmarked spins are unspecified. The middle left box represents the column $C_{(1,1)}$ and the bottom vertex of the middle box is unstable.}
\label{fig: 32fig_a}
\end{center}
\end{figure}

Now we claim that almost surely, for all large times at which $A_3^+$ occurs, the spins at $(3,1,1)$ and $(3,1,2)$ must be $+1$. For suppose that this is false; that is, with positive probability, infinitely often both $A_3^+$ occurs and at least one of these spins is $-1$. Because $(2,1,0)$ is unstable, it can flip to $-1$, and an application of Lemma~\ref{22} shows that with positive probability, the columns $C_{(2,1)}, C_{(3,1)}, C_{(2,2)}$ and $C_{(3,2)}$ would have a configuration described in Remark~\ref{rem: neg31} infinitely often, a contradiction. This means that $A_4^+ := A_3^+ \cap \{ \sigma_{(3,1,1)}=\sigma_{(3,1,2)}=+1\}$ occurs infinitely often with positive probability.

The spins at $(3,1,1)$ and $(3,1,2)$ now have at least two negative neighbors. If they have at least three, then they can flip, forcing one of $(2,1,1)$ or $(2,1,2)$ to be $+1$ but unstable. If this occurred infinitely often with positive probability, then it would contradict Lemma~\ref{30}, since $(2,1,0)$ is a neighboring $+1$ unstable vertex. Therefore $A_5^+= A_4^+  \cap \{\sigma_{(3,0,1)} = \sigma_{(3,0,2)} = +1 \}$ occurs infinitely often with positive probability.

We now invoke Lemma~\ref{31}. If, with positive probability, for infinitely many of the times at which $A_5^+$ occurs, either the spin at $(2,0,1)$ or $(2,0,2)$ were equal to $+1$ then the columns $C_{(2,1)}, C_{(3,1)}, C_{(2,0)}$ and $C_{(3,0)}$ would have a configuration described in that lemma. Therefore $A_6^+ = A_5^+ \cap \{C_{(2,0)} \text{ is negative}\}$ occurs infinitely often with positive probability. But now if $\sigma_{(2,1,0)}$ flips to $-1$, the other two spins in $C_{(2,1)}$ are $+1$ and unstable, and Lemma~\ref{22} says this will occur infinitely often with positive probability, contradicting Lemma~\ref{30}.

{\bf Case 2.} $A_2^-$ occurs infinitely often with positive probability. 
We first claim that almost surely, for all large times at which $A_2^-$ occurs, the spins at $(2,1,1)$ and $(2,1,2)$ must each have at least three negative neighbors not contained in $C_{(2,1)}$. To see this, suppose for a contradiction that with positive probability, infinitely often both $A_2^-$ occurs and one of these spins (by symmetry, we can say $\sigma_{(2,1,1)}$) has at most two negative neighbors not contained in $C_{(2,1)}$. Because $\sigma_{(2,1,0)}$ is unstable, it can flip to $+1$ and by Lemma~\ref{22}, this will occur infinitely often almost surely. After this flip, $\sigma_{(2,1,1)}$ is then unstable and negative, so can flip to $+1$. This leaves $\sigma_{(2,1,2)}$ negative and unstable (and therefore, for all large times, with exactly three positive and three negative neighbors). After $\sigma_{(2,1,2)}$ flips to $+1$, then $C_{(2,1)}$ is positive with $\sigma_{(2,1,2)}$ unstable. Using Lemma~\ref{22} and permuting levels 0 and 2 in the slab shows that $A_2^+$ occurs infinitely often with positive probability, contradicting case 1.


Therefore $A_3^- : = A_2^- \cap \{C_{(3,1)} \text{ and } C_{(2,0)} \text{ are negative in levels } 1 \text{ and } 2\}$ occurs infinitely often with positive probability. Further, as $\sigma_{(2,1,0)}$ is unstable on $A_3^-$ and already has three negative neighbors, the spins at $(3,1,0)$ and $(2,0,0)$ must be $+1$. This results in the sign distribution displayed in Figure~\ref{fig: 32fig_c}.

\begin{figure}[h]
\begin{center}
\setlength{\unitlength}{1.4cm}
\begin{picture}(3,3.5)
\multiput(0,0)(1,0){4}{\line(0,1){3}}
\multiput(0,0)(0,1){4}{\line(1,0){3}}
\multiput(.2, 1.2)(0,1){2}{$+$} \multiput(.45, 1.45)(0,1){2}{$+$} \multiput(.7, 1.7)(0,1){2}{$+$}
\put(1.2, 1.2){$-$} \put(1.45, 1.45){$-$} \put(1.7, 1.7){$-$}
\put(2.2, 1.2){$+$} \put(2.45, 1.45){$-$} \put(2.7,1.7){$-$} 
\put(1.2,.2){$+$} \put(1.45,.45){$-$} \put(1.7,.7){$-$}
\put(1.2, 2.2){$-$} \put(1.45, 2.45){$-$}\put(1.7, 2.7){$-$}
\end{picture}
\caption{Illustration of the event $A_3^-$ in Lemma~\ref{32}. All unmarked spins are unspecified. The middle left box represents the column $C_{(1,1)}$ and the bottom vertex of the middle box is unstable.}
\label{fig: 32fig_c}
\end{center}
\end{figure}

On the event $A_3^-$, if there is any negative spin in $C_{(3,2)}$, then the columns $C_{(2,1)}, C_{(2,2)}, C_{(3,1)}$ and $C_{(3,2)}$ would have a configuration described in Remark~\ref{rem: neg31}, so almost surely, for all large times at which $A_3^-$ occurs, the column $C_{(3,2)}$ must be positive. However, $\sigma_{(2,1,0)}$ can flip to $+1$, leaving $\sigma_{(2,2,0)}$ unstable, and flipping to $+1$, leaving both $\sigma_{(2,2,1)} $ and $\sigma_{(2,2,2)}$ unstable and negative. Lemma~\ref{22} implies this will occur infinitely often with positive probability and this contradicts Lemma~\ref{30}.
\end{proof}

\begin{proof}[Proof of Proposition~\ref{12}] 
By translation invariance and symmetry we can take $u=(1,1)$ and $v=(2,1)$. Since $p=1/2$, it is enough to show that almost surely, when $C_{(1,1)}$ fixates to $+1$, then $C_{(2,1)}$ also fixates (either to $+1$ or to $-1$).

We first prove that almost surely, if $C_{(1,1)}$ fixates to $+1$ then for all large times, each vertex in this column must have at least 2 stable neighbors outside $C_{(1,1)}$ with spin $+1$. If this were false, then with positive probability there would be infinitely many times at which some spin (say at $(1,1,0)$) has at least three neighbors outside $C_{(1,1)}$ which are either not positive or not stable. Note that none of these neighbors are neighbors of each other, so the unstable ones will all be unstable even if any of them flip. Since they have a positive probability to flip, Lemma~\ref{22} implies that with positive probability, on the event that $C_{(1,1)}$ fixates to $+1$, the spin at $(1,1,0)$ will have at least three negative neighbors. Another application of Lemma~\ref{22} implies that this spin will flip infinitely often, a contradiction since it fixates.

Therefore $B$ occurs infinitely often with positive probability, where
\[
B = \left\{\begin{array}{c}
C_{(1,1)} \text{ is positive and each of its spins have at least } \\
\text{ two positive stable neighbors outside } C_{(1,1)} \end{array} \right\}\ .
\]

We next claim that almost surely, if $C_{(1,1)}$ fixates to $+1$ then $C_{(2,1)}$ is positive infinitely often or negative infinitely often. If this is not the case, then with positive probability, $C_{(1,1)}$ fixates to $+1$ and $C_{(2,1)}$ has exactly two like spins for all large times. By Proposition~\ref{11}, $C_{(2,1)}$ cannot fixate, so it hs exactly two positive spins infinitely often. But then the negative spin is unstable, and Lemma~\ref{22} implies that $C_{(2,1)}$ is positive infinitely often, a contradiction.

If $C_{(2,1)}$ does not fixate (with positive probability) then it must flip infinitely often, so the previous paragraph implies that almost surely, if $C_{(1,1)}$ fixates then infinitely often $C_{(2,1)}$ will both be monochromatic and have an unstable spin. By spatial symmetry we may assume that $\sigma_{(2,1,0)}$ is unstable. So far we have shown that at least one of $B^+$ or $B^-$ occurs infinitely often with positive probability, where $B^+ = B \cap \{C_{(2,1)} \text{ is positive and } \sigma_{(2,1,0)} \text{ is unstable}\}$ and $B^-$ is the same event with positive replaced by negative.

{\bf Case 1.} $B^+$ occurs infinitely often with positive probability. Because $\sigma_{(2,1,0)}$ is unstable and already has three positive neighbors, almost surely for all large times at which $B^+$ occurs, the spins at $(2,2,0), (3,1,0)$ and $(2,0,0)$ must be negative. See Figure~\ref{fig: prop_4_1_a}.

\begin{figure}[h]
\begin{center}
\begin{picture}(3,4)
\multiput(0,0)(1,0){4}{\line(0,1){3}}
\multiput(0,0)(0,1){4}{\line(1,0){3}}
\put(.2, 1.2){$+$} \put(.45, 1.45){$+$} \put(.7, 1.7){$+$}
\put(1.2, 1.2){$u$} \put(1.45, 1.45){$+$} \put(1.7, 1.7){$+$}
\put(1.2, .2){$-$}\put(2.2, 1.2){$-$}\put(1.2, 2.2){$-$}
\end{picture}
\end{center}
\caption{Illustration of the event $B^+$ at large times, where the middle left box represents $C_{(1,1)}$. The neighbors of the unstable spin at $(2,1,0)$ (labeled $u$) outside of $C_{(1,1)}$ and $C_{(2,1)}$ must be negative.} 
\label{fig: prop_4_1_a}
\end{figure}

Two neighbors of $(1,1,0)$ outside of $C_{(1,1)}$ must be positive and stable, so at least one must be in the set $\{\sigma_{(1,2,0)}, \sigma_{(1,0,0)}\}$. By spatial symmetry we may assume that $\sigma_{(1,2,0)}$ is positive and stable on $B^+$ infinitely often with positive probability. This means that at least one other spin at a vertex in $C_{(1,1)}$ is positive. The remaining spin has at least 3 positive neighbors and by Lemma~\ref{22} will be $+1$ infinitely often. This means that with positive probability, infinitely often on $B^+$, the column $C_{(1,2)}$ is positive. If there is a positive spin in $C_{(2,2)}$, then $C_{(1,1)}, C_{(1,2)}, C_{(2,1)}$ and $C_{(2,2)}$ contain a configuration described in Lemma~\ref{31}, so for all large times $C_{(2,2)}$ is negative. Lemma~\ref{32} then implies that $B^+$ is e. absent.
 
{\bf Case 2.} $B^-$ occurs infinitely often with positive probability. When $B^-$ occurs, $(2,1,0)$ has three positive neighbors, so either $(2,0,0)$ or $(2,2,0)$ (or both) are positive. By symmetry we may assume that let $B^- \cap \{\sigma_{(2,2,0)}=+1\}$ occurs infinitely often with positive probability. Because each spin in $C_{(1,1)}$ has at least two positive stable neighbors outside of $C_{(1,1)}$, either $C_{(1,2)}$ or $C_{(1,0)}$ must contains at least two positive spins. Just as before, if $C_{(1,2)}$ contains two positive spins, the other has at least three positive neighbors and Lemma~\ref{22} implies the columns will be positive infinitely often. The same holds for $C_{(1,0)}$. These two cases will complete the proof below. We first consider the case that $B^- \cap \{\sigma_{(2,2,0)}=+1, C_{(1,2)} \text{ is positive}\}$ occurs infinitely often, and the configuration is shown in Figure~\ref{fig: prop_4_1_b}. 
\begin{figure}[h]
\begin{center}
\begin{picture}(3,4)
\multiput(0,0)(1,0){4}{\line(0,1){3}}
\multiput(0,0)(0,1){4}{\line(1,0){3}}
\put(.2, 1.2){$+$} \put(.45, 1.45){$+$} \put(.7, 1.7){$+$}
\put(.2, 2.2){$+$} \put(.45, 2.45){$+$} \put(.7, 2.7){$+$}
\put(1.2, 1.2){$u$} \put(1.45, 1.45){$-$} \put(1.7, 1.7){$-$}
\put(1.2,2.2){$+$}
\end{picture}
\end{center}
\caption{Illustration of the event $B^-$ in the case that $C_{(1,2)}$ is positive. The middle left box represents $C_{(1,1)}$ and the spin $\sigma_{(2,2,0)}$ is positive.} 
\label{fig: prop_4_1_b}
\end{figure}
By lemma \ref{32} the positive vertex $(2,2,0)$ must be stable, hence $C_{(2,2)}$ must contain at least two positive spins. But then the unstable spin $\sigma_{(2,1,0)}$ can flip to $+1$, giving a configuration in these four columns described in Lemma~\ref{31}. This occurs infinitely often with positive probability, a contradiction.

The other possibility is that $B^- \cap \{\sigma_{(2,2,0)} \cap C_{(1,0)} \text{ is positive}\}$ occurs infinitely often with positive probability. If the spin $\sigma_{(2,0,0)}$ is positive at infinitely many of these times, then we have a configuration symmetrical to that in the previous paragraph, leading to a contradiction. Otherwise $\sigma_{(2,0,0)}$ is negative at all such large times. The configuration is displayed in Figure~\ref{fig: prop_4_1_c}.
\begin{figure}[h]
\begin{center}
\begin{picture}(3,4)
\multiput(0,0)(1,0){4}{\line(0,1){3}}
\multiput(0,0)(0,1){4}{\line(1,0){3}}
\put(.2, 1.2){$+$} \put(.45, 1.45){$+$} \put(.7, 1.7){$+$}
\put(.2, .2){$+$} \put(.45, .45){$+$} \put(.7, .7){$+$}
\put(1.2, 1.2){$u$} \put(1.45, 1.45){$-$} \put(1.7, 1.7){$-$}
\put(1.2,.2){$-$}
\end{picture}
\end{center}
\caption{Illustration of the event $B^-$ in the case that $C_{(1,0)}$ is positive. The middle left box represents $C_{(1,1)}$ and the spin $\sigma_{(2,0,0)}$ is positive.} 
\label{fig: prop_4_1_c}
\end{figure}
Again, if $C_{(2,0,0)}$ has another vertex with a positive spin infinitely often, Lemma~\ref{31} gives a contradiction after $\sigma_{(2,1,0)}$ flips to $+1$, so $C_{(2,0,0)}$ is negative for all large times. Applying Lemma \ref{32} to $C_{(1,0)}, C_{(1,1)}, C_{(2,1)}$ and $C_{(2,0)}$, we obtain a contradiction.

\end{proof}

\section{Proof of Theorem~\ref{02}}\label{sec: thm_3}

For any $z \in \mathbb{Z}^2$, 
\[
\mathbb{P}(C_z \text{ fixates}) \geq [\min\{p, (1-p)\}]^{12}\ ,
\]
since the initial configuration may have all spins in $C_z, C_{z+(1,0)}, C_{z+(0,1)}$ and $C_{z+(1,1)}$ of the same sign. By translation-ergodicity, almost surely there exist columns that fixate. If not all columns fixate, we may almost surely find neighboring columns, one which fixates and one which doesn't. By countability, there exist neighboring columns $C_u$ and $C_v$ that have positive probability for $C_u$ to fixate but for $C_v$ not to and this contradicts Proposition~\ref{11}.
\qed

\section{ Proof of Theorem \ref{03}}\label{sec: thm_4}





\begin{figure}[h]
\begin{center}
\setlength{\unitlength}{1.2cm} 
\begin{picture}(10,10)
\multiput(-2,0)(1,0){15}{\line(0,1){11}}
\multiput(-2,0)(0,1){12}{\line(1,0){14}}

\multiput(.2, 2.2) (1,0) {6}{$+$}
\multiput(.2, 2.7) (1,0) {6}{$+$}
\multiput(.7, 2.2) (1,0) {2}{$+$}\multiput(.7, 2.7) (1,0) {2}{$+$}
\multiput(4.7, 2.7) (1,0) {6}{$+$}\multiput(4.7, 2.2) (1,0) {2}{$+$}
\multiput(6.2, 2.7) (1,0) {4}{$+$}
\multiput(8.2, 2.2) (1,0) {2}{$+$}\multiput(8.7, 2.2) (1,0) {2}{$+$}

\multiput(.2, 3.2) (1,0) {6}{$+$}
\multiput(.2, 3.7) (1,0) {6}{$+$}
\multiput(.7, 3.2) (1,0) {2}{$+$}\multiput(.7, 3.7) (1,0) {2}{$+$}
\multiput(4.7, 3.7) (1,0) {6}{$+$}\multiput(4.7, 3.2) (1,0) {2}{$+$}
\multiput(6.2, 3.7) (1,0) {4}{$+$}
\multiput(8.2, 3.2) (1,0) {2}{$+$}\multiput(8.7, 3.2) (1,0) {2}{$+$}

\multiput(2.7, .2)(0,1){11}{$-$}\multiput(2.7, .7)(0,1){11}{$-$}
\multiput(3.7, .2)(0,1){11}{$-$}\multiput(3.7, .7)(0,1){11}{$-$}
\multiput(6.7, .2)(0,1){11}{$-$}\multiput(6.2, .2)(0,1){11}{$-$}
\multiput(7.7, .2)(0,1){11}{$-$}\multiput(7.2, .2)(0,1){11}{$-$}

\multiput(2.2, .2)(0,1){11}{$-$}\multiput(2.2, .7)(0,1){2}{$-$}
\multiput(3.2, .2)(0,1){11}{$-$}\multiput(3.2, .7)(0,1){2}{$-$}
\multiput(6.7, .7)(0,1){11}{$-$}\multiput(6.2, .7)(0,1){2}{$-$}
\multiput(7.7, .7)(0,1){11}{$-$}\multiput(7.2, .7)(0,1){2}{$-$}

\multiput(-1.8, 4.2)(0,1){3}{$-$}\multiput(-1.3, 4.2)(0,1){3}{$-$}
\multiput(-1.8, 4.7)(0,1){3}{$-$}\multiput(-1.3, 4.7)(0,1){3}{$-$}
\multiput(-.8, 4.2)(0,1){3}{$-$}\multiput(-.3, 4.2)(0,1){3}{$-$}
\multiput(-.8, 4.7)(0,1){3}{$-$}\multiput(-.3, 4.7)(0,1){3}{$-$}

\multiput(.2, 4.2)(1,0){8}{$-$}\multiput(.7, 4.2)(1,0){12}{$-$}
\multiput(.2, 4.7)(1,0){10}{$+$}

\multiput(.2, 5.7)(0,1){4}{$+$}\multiput(.7, 4.7)(0,1){5}{$+$}
\multiput(1.2, 5.7)(0,1){4}{$+$}\multiput(1.7, 4.7)(0,1){5}{$+$}

\multiput(1.2, 7.2)(0,1){2}{$+$}\multiput(1.7, 7.2)(0,1){2}{$+$}
\multiput(.2, 7.2)(0,1){2}{$+$}\multiput(.7, 7.2)(0,1){2}{$+$}

\multiput(1.2, 5.2)(0,1){2}{$-$}\multiput(1.7, 5.2)(0,1){2}{$-$}
\multiput(.2, 5.2)(0,1){2}{$-$}\multiput(.7, 5.2)(0,1){2}{$-$}

\multiput(4.7, 4.7)(1,0){8}{$-$}\multiput(10.7, 4.7)(1,0){2}{$-$}
\multiput(8.2, 4.2)(1,0){2}{$+$}\multiput(10.2, 4.7)(1,0){2}{$-$}\multiput(8.2, 4.2)(1,0){4}{$-$}
\multiput(2.2, 5.7)(1,0){8}{$+$}\multiput(8.2, 5.2)(1,0){2}{$+$}

\multiput(8.7, 5.7)(1,0){4}{$-$}\multiput(8.7, 5.2)(1,0){4}{$-$}
\multiput(10.2, 5.7)(1,0){2}{$-$}\multiput(10.2, 5.2)(1,0){2}{$-$}

\multiput(8.2, 6.2)(0,1){3}{$+$}\multiput(8.7, 6.2)(0,1){3}{$+$}
\multiput(8.2, 6.7)(0,1){3}{$+$}\multiput(8.7, 6.7)(0,1){3}{$+$}
\multiput(9.2, 6.2)(0,1){3}{$+$}\multiput(9.7, 6.2)(0,1){3}{$+$}
\multiput(9.2, 6.7)(0,1){3}{$+$}\multiput(9.7, 6.7)(0,1){3}{$+$}

\multiput(2.2, 9.7)(0,1){2}{$-$}\multiput(3.2, 9.7)(0,1){2}{$-$}
\multiput(6.2, 9.7)(0,1){2}{$-$}\multiput(7.2, 9.7)(0,1){2}{$-$}

\multiput(4.2, 7.2)(0,1){2}{$+$}\multiput(4.2, 7.7)(0,1){2}{$+$}
\multiput(4.7, 7.2)(0,1){2}{$+$}\multiput(4.7, 7.7)(0,1){2}{$+$}
\multiput(5.2, 7.2)(0,1){2}{$+$}\multiput(5.2, 7.7)(0,1){2}{$+$}
\multiput(5.7, 7.2)(0,1){2}{$+$}\multiput(5.7, 7.7)(0,1){2}{$+$}

\multiput(6.2, 6.7)(0,1){3}{$+$}\multiput(7.2, 6.7)(0,1){3}{$+$}
\multiput(6.7, 7.7)(0,1){2}{$+$}\multiput(7.7, 7.7)(0,1){2}{$+$}

\multiput(2.2, 6.7)(0,1){3}{$+$}\multiput(3.2, 6.7)(0,1){3}{$+$}
\multiput(2.2, 7.2)(0,1){2}{$+$}\multiput(3.2, 7.2)(0,1){2}{$+$}

\multiput(4.2, 6.7)(1,0){2}{$+$}
\multiput(4.2, 6.2)(1,0){2}{$-$}
\multiput(4.7, 6.2)(1,0){2}{$-$}
\multiput(4.7, 6.7)(1,0){2}{$-$}

\multiput(4.2, 5.2)(1,0){2}{$-$}
\multiput(4.7, 5.2)(1,0){2}{$-$}
\multiput(4.7, 5.7)(1,0){2}{$-$}

\linethickness{.05 cm}
\multiput(2,4)(0,3){2}{\line(1,0){6}}
\multiput(2,4)(6,0){2}{\line(0,1){3}}

\linethickness{.09 cm}
\multiput(4,6)(0,1){2}{\line(1,0){4}}
\multiput(4,6)(4,0){2}{\line(0,1){1}}
\put(7.5, 6.4){$1$}
\put(6.5, 6.4){$2$}
\put(5.5, 6.4){$3$}
\put(4.5, 6.4){$4$}
\end{picture}
\end{center}
\caption{Construction showing existence of non-fixating sites within the bold lines. Each unit box represents a column and the spins in a box begin with level 3 in the top left and proceed counter-clockwise to level 0 in the top right. All specified spins outside thick bold lines are fixed. The level 0 spins in the thick bold lines begin $-1$ and flip infinitely often.}
\label{fig: 4_case}
\end{figure}

Here we will show that $S_4$ does not fixate. This is a proof by example and is illustrated in Figure~\ref{fig: 4_case}. The notation used in the figure is as follows. Each unit square represents a column $C_{(x,y)}$ in $S_4$. The spin of vertex $(x,y,3)$ is shows at the top-left of the square, with $(x,y,2), (x,y,1)$ and $(x,y,0)$ proceeding counter-clockwise. With $C_{(0,0)}$ the column at the bottom left of the figure, let $A_{(0,0)}$ be the event (in $\Omega_4$) that all spins in the box $[0,13]\times [0,11]\times \{0, \ldots, 3\}$ have values as shown in Figure~\ref{fig: 4_case} (with blank spins unspecified). The reader may check that (a) all sites within the medium-line box (outside the bold box) are fixed with one positive and three negative spins, (b) all specified spins outside the medium-line box are fixed and (c) the spins in the bold box are fixed except for those with third coordinate 0 (that is, those pictured in the upper right of the unit boxes) flip infinitely often. The flipping spins begin as all with value $-1$ and flip right to left from $-1$ to $+1$ as denoted by the numbering. Once they have flipped from $-1$ to $+1$ they flip back in the reverse order.

The event $A_{(0,0)}$ has positive probability and by translation ergodicity, almost surely some translate of it occurs. So with probability one, there exist spins which flip infinitely often.


\appendix
\section{$S_2$ fixation under periodic boundary conditions}
In this appendix we give an alternative proof (to the one in \cite{DKNS}) of fixation in $S_2$ with periodic boundary conditions. As the arguments follow the same lines as those presented in this paper, we keep the proofs concise. We use a similar notation to the $S_3$ slab; the column $C_{(x,y)}$ consists of the pair of spins with the first two coordinates $(x,y)$. $C_{(x,y)}$ is positive (negative) if both of its spins are positive (negative). Note that, due to the boundary condition, the edge between $(x,y,0)$ and $(x,y,1)$ counts twice in the energy computation \eqref{eq: energy} of either site.

The proof structure is identical to the $S_3$ case and directly combines analogous Propositions \ref{71} and \ref{72}.
As before, the first proposition shows that fixed columns are monochromatic. The proof is simple and is contained in \cite{DKNS}.
\begin{proposition} \label{71}
With probability one, if a column in $S_2$ flips finitely often then it is eventually monochromatic.
\end{proposition}

The second proposition shows that neighbors of fixed columns are fixed.
\begin{proposition}\label{72}
Let $u,v \in \mathbb{Z}^2$ be neighbors. With probability one, if $C_u$ flips finitely often, then so does $C_v$.
\end{proposition}
We will use the following lemma repeatedly in the proof:
\bl \label{S2_1}
Let $A$ be the event that $C_{(1,1)}$ is positive, each of the columns $C_{(1,2)}$,  $C_{(2,1)}$ and $C_{(2,2)}$ contains at least one positive spin, and at least one of these columns contains a negative spin. $A$ is e.absent.
\el
\begin{remark}
By identical reasoning, the lemma holds with all positive spins replaced by negative.
\end{remark}

\begin{proof} 
If $A$ occurs infinitely often with positive probability then at each occurrence of $A$, all non-positive spins among the four columns have at least 3 positive neighbors (counting spins in the same column twice) and thus have positive probability to flip to $+1$. By Lemma \ref{22}, all will flip to  $+1$ and fixate, and this is a contradiction.
\end{proof}


\begin{proof}[Proof of Proposition~\ref{72}]
By translation invariance and symmetry, we can take $u=(1,1)$ and $v=(2,1)$. We will show that almost surely, when $C_{(1,1)}$ fixates to $+1$, then $C_{(2,1)}$ also fixates (either to $+1$ or $-1$). The case in which $C_{(1,1)}$ fixates to $-1$ is handled similarly.

We will assume the contrary, that with positive probability $C_{(1,1)}$ fixates to $+1$ but a vertex in $C_{(2,1)}$ flips infinitely often. It must then be unstable infinitely often. By spatial symmetry and Lemma~\ref{22}, we will assume this unstable spin to be $\sigma_{(2,1,0)}$ and take it to be positive. As in the proof of Proposition~\ref{12}, almost surely, if $C_{(1,1)}$ fixates to $+1$ then for all large times, each vertex in this column must have at least 2 stable neighbors outside $C_{(1,1)}$ with spin $+1$. The proof is exactly as before -- if not, then a spin of $C_{(1,1)}$ has at least three unstable neighbors which can flip to $-1$ and then force it, by Lemma~\ref{22}, to flip. By spatial symmetry then, the event
\[
B = \{C_{(1,1)} \text{ is positive, }\sigma_{(2,1,0)} \text{ is positive unstable and } \sigma_{(1,0,0)} = +1\}
\]
occurs infinitely often with positive probability. Define $B^+ = B \cap \{\sigma_{(2,1,1)}=+1\}$ and $B^- = B \cap \{\sigma_{(2,1,1)}=-1\}$. We give two cases.



{\bf Case 1.} $B^+$ occurs infinitely often with positive probability. Almost surely for all large times at which $B^+$ occurs, $C_{(2,0)}$ must be negative; this follows from Lemma~\ref{S2_1}. Furthermore at all such large times $\sigma_{(2,1,0)}$ has exactly three positive and three negative neighbors. This implies that $\sigma_{(3,1,0)} = \sigma_{(2,2,0)} = -1$. Further, both of $\sigma_{(3,1,1)}$ and $\sigma_{(2,2,1)}$ must be positive, for if either were negative, $\sigma_{(2,1,0)}$ could flip to $-1$, leaving $\sigma_{(2,1,1)}$ with 4 negative neighbors and an energy lowering flip. We can now apply Lemma~\ref{S2_1} again to both blocks of columns $C_{(1,1)}, C_{(2,1)}, C_{(1,2)}, C_{(2,2)}$ and $C_{(2,1)}, C_{(2,2)}, C_{(3,1)}, C_{(3,2)}$ respectively to deduce that $C_{(1,2)}$ and $C_{(3,2)}$ are negative. But this leaves $\sigma_{(2,2,1)}$ positive with at least 4 negative neighbors, so it can make an energy lowering flip, a contradiction for large times.

{\bf Case 2.} $B^-$ occurs infinitely often with positive probability. Again by Lemma \ref{S2_1}, $C_{(2,0)}$ must almost surely be negative at all large times at which $B^-$ occurs. The spin $\sigma_{(2,1,0)}$ must have three positive neighbors at large times, so $\sigma_{(2,2,0)}=\sigma_{(3,1,0)}=+1 $.  By Lemma \ref{S2_1} again applied to $C_{(1,1)}, C_{(2,1)}, C_{(1,2)}$ and $C_{(2,2)}$, the column $C_{(1,2)}$ is negative for all large times at which $B^-$ occurs. As above, both spins $\sigma_{(3,1,1)}$ and $\sigma_{(2,2,1)}$ must be positive, lest $\sigma_{(2,1,0)}$ flips to $+1$ and giving $\sigma_{(2,1,1)}$ an energy lowering flip. But now if $\sigma(2,1,0)$ flips to $-1$, the positive spin at $(2,2,0)$ has at least four negative neighbors, so it can make an energy lowering flip, a contradiction for large times.
\end{proof}


\noindent
{\bf Acknowledgements.} M.D. and V.S. thank the Courant Institute for funds and hospitality while some of this work was done. We also thank
Leonardo T. Rolla for fruitful discussions and Sidney Redner for useful communications. Work of V.S. was supported by ESF RGLIS network.

\end{document}